\documentclass[12pt,a4paper]{amsart}

\usepackage{amscd}

\usepackage{amssymb,latexsym}

\usepackage{fixltx2e}

\usepackage[all]{xy}

\usepackage{array}
\usepackage{geometry}

\newtheorem{thm}{Theorem}[section]

\newtheorem{cor}[thm]{Corollary}

\newtheorem{lem}[thm]{Lemma}

\newtheorem{prop}[thm]{Proposition}

\theoremstyle{definition}

\theoremstyle{remark}

\newtheorem{rem}[thm]{Remark}

\newtheorem*{rem*}{Remark}

\numberwithin{equation}{section}

\newcommand{\ovl}[1]{\overline{#1}}

\newcommand{\mQ}{\mathbb{Q}}

\newcommand{\ra}{\rightarrow}

\newcommand{\fp}{\mathfrak{p}}

\def\divides{{\,|\,}}

\def\ndivides{\nmid}

\newcommand{\field}[1]{\mathbb{#1}}

\newcommand{\Q}{\field{Q}}

\newcommand\oline[1] {{\overline{#1}}}

\newcommand{\PP}{{\mathfrak P}}

\newcommand{\fC}{{\mathfrak C}}

\DeclareMathOperator{\Gal}{Gal}

\DeclareMathOperator{\cd}{cd}

\DeclareMathOperator{\Br}{Br}

\DeclareMathOperator{\TBr}{TBr}

\DeclareMathOperator{\ind}{ind}

\DeclareMathOperator{\Hom}{Hom}

\DeclareMathOperator{\image}{Im}

\DeclareMathOperator{\charak}{char}

\DeclareMathOperator{\lcm}{lcm}

\newcommand{\grA}{\mathsf{A}}

\newcommand{\grB}{\mathsf{B}}

\newcommand{\grC}{\mathsf{C}}

\newcommand{\grF}{\mathsf{F}}

\newcommand{\grD}{\mathsf{D}}

\newcommand{\grE}{\mathsf{E}}

\newcommand{\grK}{\mathsf{K}}

\newcommand{\grL}{\mathsf{L}}

\newcommand{\grM}{\mathsf{M}}

\newcommand{\grT}{\mathsf{T}}

\newcommand{\grU}{\mathsf{U}}

\newcommand{\grZ}{\mathsf{Z}}

\DeclareMathOperator{\gr}{\operatorname{\mathsf{gr}}}

\DeclareMathOperator{\Aut}{\operatorname{\mathsl{Aut}}}

\def\hsp{\mspace{1mu}}

\newcommand{\DIM}[2]{[#1{\hsp:\hsp}#2]}

\DeclareMathAlphabet{\mathsl}{OT1}{cmr}{m}{sl}

\renewcommand{\int}{\operatorname{\mathsl{int}}}

\begin{document}

\begin{abstract} We show that a finite-dimensional tame division
algebra $D$ over a Henselian field $F$ has a maximal subfield Galois over
$F$ if and only if its residue division algebra $\oline D$ has a
maximal subfield Galois over the residue field $\oline F$.

This generalizes
the mechanism behind several known noncrossed product constructions to a crossed product criterion for all
tame division algebras, and in particular for all division algebras if
the residue characteristic is~$0$.
If  $\oline F$ is a global field, the criterion leads to a description of the location of noncrossed products among tame division algebras, and their discovery in new parts of the Brauer group.
\end{abstract}

\def\Tech{Department of Mathematics, Technion --- Israel Institute of Technology, Haifa 32000, Israel}

\author{Timo Hanke}

\address{
Lehrstuhl D f\"ur Mathematik\\
RWTH Aachen\
Templergraben 64\\
D-52062 Aachen\\
Germany}

\email{hanke@math.rwth-aachen.de}

\author{Danny Neftin}
\address{
Department of Mathematics,
University of Michigan, Ann Arbor,
530 Church St.,
Ann Arbor, MI 48109-1043
USA}

\email{neftin@umich.edu}

\author{Adrian Wadsworth}

\address{
Department of Mathematics,
University of California, San Diego,
9500 Gilman Drive, La Jolla,
California 92093-0112 USA}

\email{arwadsworth@ucsd.edu}

\title{Galois subfields of tame division algebras}



\maketitle

\section{Introduction}

A division algebra  $D$ finite-dimensional over its center $F$
is called a {\it crossed product} if it
contains a maximal subfield which is Galois over~$F$; otherwise it is
a {\it noncrossed product}.
The question of existence of noncrossed products arose in the
1930's, and was answered affirmatively by Amitsur in 1972
\cite{amitsur:central-div-alg}.
Subsequently, their existence over more familiar fields $F$ has been
 studied by many authors, see for example
\cite{jacob-wadsworth:constr-noncr-prod, tignol:malcev, brussel:noncr-prod, hanke:twisted, BMT:p-adic-curves}.


Recall that for a Henselian valued field $F$, the valuation of~$F$ extends uniquely to a valuation
of $D$ for every  finite-dimensional division algebra~$D$ over~$F$, see for example \cite[\S1]{jacob-wadsworth:div-alg-hensel-field}.
In case $D$ is inertially split, i.e.\ split by an unramified extension
of $F$,
it was  known long ago that $D$ is a crossed product only if the residue
division algebra $\ovl D$ is a crossed product
\cite[Thm.~5.15(b)]{jacob-wadsworth:div-alg-hensel-field}.
This criterion traces back to Saltman~\cite{saltman:noncr-prod-small-exp} who
used it  to construct new noncrossed products of higher
index from ones already known.
By a more complete criterion \cite{hanke:maxsf},
any inertially split $D$ is a crossed product if and only if
$\ovl D$ contains a maximal subfield Galois over the residue field
$\ovl F$.
Note that this is a stronger condition than saying $\ovl D$ is a
crossed product since $\ovl F$ may be a proper subfield of the center
of~$\ovl D$.
This criterion goes back to Brussel \cite{brussel:noncr-prod}, who
used such an argument over complete rank $1$ valued fields to
obtain noncrossed products over fields as elementary as $\Q((t))$.
Subsequently, the criterion has led to a description of
the ``location'' 
of  crossed and noncrossed products among all inertially split division
algebras over $F$,
when $F$ is Henselian with global residue field,
\cite{hanke-sonn:location, hns:existence-bounds}.







In this paper we consider the larger class of division algebras $D$
which are {\it tamely ramified} (or
 {\it tame} for short) over their center a Henselian field $F$,
i.e. split by a tamely ramified extension of $F$. In particular, these
include all division algebras whose degree is prime to the residue
characteristic.
Our main result, Theorem \ref{main.thm}, generalizes the
criterion mentioned above 
to tame division algebras:
\begin{thm}\label{main.thm} Let $F$ be a Henselian field, and  $D$  a
 finite-dimensional tamely ramified division algebra with center $F$.
Then $D$ has a maximal subfield Galois over $F$ if and only if
$\oline{D}$ has a maximal subfield Galois over $\oline F$.
\end{thm}
Theorem \ref{main.thm} is useful in determining which tame division
algebras are crossed products over all Henselian fields $F$ whose
residue field is sufficiently well understood. When $\oline F$ has
 cohomological dimension $1$ or is a local field, we deduce that
all finite-dimensional tame division algebras with center $F$ are crossed products.  When
$\oline F$ is a global field we describe the location of noncrossed
products among tame division algebras, extending
\cite{hanke-sonn:location} and \cite{hns:existence-bounds}, and
locating noncrossed products in new parts of the Brauer~group, see \S\ref{sec:global-residue}.

The main difficulty in proving Theorem \ref{main.thm} lies in the
``only if" implication.
Given a maximal subfield $M$ of $D$, Galois over $F$, neither the
residue field $\oline M$ itself
nor the compositum of $\ovl M$ with the center of~$\ovl D$ need to
be a maximal subfield of $\oline D$.
Hence in the tame case the construction is significantly different
from the inertially split case, where it was enough to consider
$\ovl M\cdot Z(\ovl D)$, 
cf. \cite{hanke:maxsf}.

Our construction of the desired maximal subfield of $\oline D$
uses the theory of graded division algebras which provides a
one-to-one correspondence between tame division algebras
over a Henselian field $F$ and  graded division
algebras over its associated graded field
$\gr(F)$~\cite{hwang-wadsworth:graded}.
It associates to $D$ a
graded division algebra $\gr(D)$  over  $\gr(F)$, and to
a maximal subfield $M$ of $D$ Galois over $F$ a
maximal graded subfield $\gr(M)$ Galois over $\gr(F)$.
Most importantly, it equips  $\gr(D)$  with canonical subalgebras  which can be entwined with  $\gr(M)$ to form  a maximal graded subfield $\grM'$ of $\gr(D)$ with
residue field which is maximal in $\oline D$ and Galois over $\oline F$.
This $\grM'$ lifts to a maximal subfield $M'$ of $D$ Galois over~$F$.



The proof of our theorem gives a good illustration of the utility
of the graded approach in working with valued division algebras.
Many properties of $D$ are faithfully reflected in $\gr(D)$, but
$\gr(D)$ has a simpler structure which is often considerably easier
to work with, as demonstrated by our use of its canonical~subalgebras. 

We thank Uzi Vishne, Eric Brussel, and Jack Sonn for helpful discussions.
We also thank Kelly McKinnie for her interest and questions which
led us to a  simplification of our proof.
This material is based upon work supported by the National Science
Foundation under Award No. DMS-1303990.

\section{Graded and valued algebras}

We first recall the basic definition and facts concerning graded
algebras, based on  \cite{hwang-wadsworth:graded} with the exception
of Section \ref{graded-fld.sec} which is based on
\cite{hwang-wadsworth:graded-fld} and
\cite{mounirh-wadsworth:semiramifed}. A more extensive treatment of
these facts will appear in~\cite{tw}.

Throughout the section we let $\Gamma$ be {\it  a torsion-free abelian
group}.

\subsection{Graded rings and division algebras} \label{13}\label{7}
A {\em graded ring $\grD$ with grade group $\Gamma$} (or a
{\em $\Gamma$-graded ring}) is
a ring with a direct sum decomposition
$\grD=\oplus_{\gamma\in \Gamma}\grD_\gamma$,
where each $\grD_\gamma$ is an additive abelian group  and
$\grD_\gamma\cdot \grD_\delta\subseteq \grD_{\gamma+\delta}$
for all $\gamma, \delta \in \Gamma$.
Set $\Gamma_\grD=\{\gamma\in\Gamma\,|\, \grD_\gamma\neq\{0\}\,\}$ and call the elements in $\grD_\gamma$, $\gamma\in \Gamma$, {\it homogenous}.
A {\it graded homomorphism} $\varphi\colon \grD \ra \grE$ of $\Gamma$-graded
rings is a homomorphism which preserves the grading, i.e.
$\varphi(\grD_\gamma)\subseteq \grE_\gamma$ for all $\gamma\in \Gamma$.

A {\it graded subring} of $\grD$ is a subring $\grE\subseteq\grD$ such
that ${\grE=\oplus_{\gamma\in \Gamma_D} ( \grD_\gamma\cap\grE)}$.
Such a decomposition defines a
$\Gamma$-grading on $\grE$ with $\grE_\gamma=\grD_\gamma\cap E$.
Note that with $\grE$ the centralizer $C_\grD(\grE)$ is also a graded subring of
$\grD$,
hence in particular the center $Z(\grD)$ is a graded subring of $\grD$.

Let $\grD$ be a graded ring with $1\neq 0$.
Then $\grD$ is said to be  a {\it graded division ring} if every
nonzero element of $\grD_\gamma$, $\gamma\in \Gamma_D$, is a unit.
In this case $\grD_0$ is a division ring, and multiplication by any nonzero element from
$\grD_\gamma$ induces an isomorphism $\grD_\gamma\cong \grD_0$ of
$\grD_0$-module; hence, $\grD_\gamma$ is a rank $1$-module over~$\grD_0$.

Commutative graded division rings are called {\it graded fields}.
If $\grD$ is a graded division ring whose center contains a graded
field $\grF$ as a graded subring, then $\grD$ is called a
{\it graded division algebra } over $\grF$.
In this case $\grD_0$ is a division algebra over~$\grF_0$.
Moreover, for a graded $\grF$-subalgebra $\grE\subseteq \grD$,  $\grD$ is free
as an $\grE$-module, cf.
\cite[Paragraph preceeding (1.6)]{hwang-wadsworth:graded},
and the dimension $[\grD:\grE]$ is defined as the rank of $\grD$ as
an $\grE$-module. We shall assume throughout that
{\it all graded division algebras are finite-dimensional}
over their centers.
The {\em degree} of a graded division algebra~$\grD$ is
$\deg\grD:=\sqrt{[\grD:Z(\grD)]}$.




Given two $\Gamma$-graded algebras $\grD$ and $\grE$ over $\grF$, the
tensor product $\grD\otimes_\grF\grE$ is also a $\Gamma$-graded algebra
with $(\grD\otimes_\grF \grE)_\gamma$ generated by all
$d_\alpha\otimes e_\beta$ where
$d_\alpha\in D_\alpha, e_\beta\in E_\beta,$ and $\alpha+\beta=\gamma$.
The double centralizer theorem  is available in the graded setting
\cite[Proposition~1.5]{hwang-wadsworth:graded} and implies, using the
same argument as in the ungraded setting, that for a graded division
algebra $\grD$ over $\grF$, $[\grM:\grF]\leq \deg\grD$ for any graded
subfield~$\grM$, with equality  if and only if $\grM$ is a maximal
graded subfield of $\grD$.


\subsection{Ramification} \label{1}

The following ramification properties of graded division rings are
analogues to ramification properties of valued division rings.
Let $\grD$ be a graded division ring, and $\grE\subseteq \grD$ a
graded division subring.
Then one easily obtains the fundamental equality,
cf.~\cite[(1.7), p.~79]{hwang-wadsworth:graded},
\begin{equation}\label{defectless.equ}
[\grD:\grE]\,=\,[\grD_0:\grE_0]\,|\Gamma_\grD:\Gamma_\grE|.
\end{equation}
We say that  $\grD$ is {\it unramified} over $\grE$ if
$\Gamma_\grD=\Gamma_\grE$ (i.e.\ $[\grD:\grE]=[\grD_0:\grE_0]$); it is
{\it totally ramified} over $\grE$ if
$[\grD:\grE]=|\Gamma_D:\Gamma_E|$ (i.e.\ $\grD_0=\grE_0$).

Let $\grA\subseteq\grD$ be a graded subring that is also a graded
division ring and contains~$\grE$.
Then $\grD/\grE$ is unramified (resp.\ totally ramified) if and only
if $\grD/\grA$ and $\grA/\grE$ are each unramified
(resp.\ totally ramified).

Let $\grF$ be a graded subfield of $Z(\grD)$, so that $\grD$ is a
graded $\grF$-division algebra.
For every $\grF_0$-subalgebra $A$ of $\grD_0$ there is a unique graded
division $\grF$-subalgebra $\grA\subseteq\grD$ with $\grA_0=A$
and $\Gamma_\grA = \Gamma_\grF$.
This $\grA$ is generated  over $\grF$ by $A$ and is canonically
isomorphic to $A\otimes_{\grF_0}\grF$.

Note that the intersection $\grA\cap\grB$ of two graded subrings of
$\grD$ is a graded subring with $(\grA\cap\grB)_\gamma=
\grA_\gamma\cap\grB_\gamma$. In the totally ramified case we have:

\begin{lem}\label{lem:tot-ram1} Let $\grA,\grB$ be two graded subrings
of $\grD$ that are also graded division rings.
 Assume $\grD$ is totally ramified over $\grB$. Then,
\begin{enumerate}
\item\label{3} $\grB=\oplus_{\gamma\in \Gamma_B} \grD_\gamma$;
\item\label{15} 
$\Gamma_{A\cap B}=\Gamma_\grA\cap\Gamma_\grB$.
\end{enumerate}
\end{lem}

\begin{proof} Clearly, $\grB\subseteq \grB':=
\oplus_{\gamma\in \Gamma_B}D_\gamma$.
For every $\gamma \in \Gamma_\grB$,
since $\grD_\gamma$ is a rank-$1$ module over $\grD_0=\grB_0$, one
has $\grB_\gamma=\grD_\gamma$.
Hence,  $\grB'=\grB$, showing~(i).

The inclusion $\Gamma_{\grA\cap\grB}\subseteq
\Gamma_\grA\cap \Gamma_\grB$ is obvious.
Conversely, for $\gamma\in \Gamma_\grA\cap\Gamma_\grB$,
using $\grB_\gamma=\grD_\gamma$,
we have $(\grA\cap\grB)_\gamma=\grA_\gamma\cap\grD_\gamma=
\grA_\gamma\neq\{0\}$,
completing~(ii).
\end{proof}




\subsection{Graded field extensions}\label{graded-fld.sec}

Let  $\grL/\grF$ be a finite extension of $\Gamma$-graded fields.
See  \cite[\S2]{hwang-wadsworth:graded-fld} for proofs of the
properties recalled in this paragraph.
As $\Gamma$ is torsion-free, $\grF$ is an integral domain and we can
form its field of quotients $q(\grF)$.
  Then,
$q(\grL) = \grL \otimes_{\grF}q(\grF)$, as ${[\grL \otimes_{\grF}q(\grF):
q(\grF)] <\infty}$ and $\grL$ has no zero divisors.
Hence, ${[\grL:\grF]=[q(\grL):q(\grF)]}$.
Moreover, the ring $\grF$ is integrally closed, and $\grL$ is the
integral closure of $\grF$ in $q(\grL)$.
In addition, for any homogeneous~$c$ in~$\grL$, the minimal
polynomial $m_{q(\grF), c}$ of $c$ over $q(\grF)$ has homogeneous
coefficients in $\grF$. In any graded field extension of
$\grF$, every root of $m_{q(\grF), c}$ is homogeneous of the
same degree as $c$, and there are graded field extensions of
$\grF$ over which $m_{q(\grF), c}$ splits.  It follows that
every $\grF$-algebra automorphism of $\grL$ as ungraded rings
actually preserves the grading on~$\grL$. Therefore, the group of
graded ring automorphisms $\Aut(\grL/\grF)$ is
canonically  isomorphic to $\Aut(q(\grL)/q(\grF))$.

We say that $\grL$ is {\it tame} (or {\em tamely ramified}) over $\grF$
if the field extension $\grL_0/\grF_0$ is separable
and  $\charak \grF_0\ndivides |\Gamma_\grL:\Gamma_\grF|$.
We say that $\grL$  is {\it normal} over~$\grF$ if for every
homogeneous $c\in \grL$, its minimal polynomial $m_{q(\grF), c}$
splits over $\grL$. More restrictively, $\grL/\grF$ is
{\it Galois} if it is Galois as an extension of ungraded commutative
rings, or, equivalently, if $\grF$ is the fixed-ring of
$\Aut(\grL/\grF)$.
These properties of $\grL/\grF$
and are equivalent to  corresponding properties of
their quotient fields, as follows:
\begin{equation}\label{table.equ}\begin{array}{|c|c|c|}
\hline
\grL/\grF & q(\grL)/q(\grF) & \text{Reference} \\
\hline
\text{tame}   & \text{separable} &
\text{\cite[Theorem~3.11(a)]{hwang-wadsworth:graded-fld}} \\
\text{normal} & \text{normal} &
\text{\cite[Lemma 1.2]{mounirh-wadsworth:semiramifed}} \\
\text{Galois} & \text{Galois} &
\text{\cite[Theorem~3.11(b)]{hwang-wadsworth:graded-fld}}\\
\hline
\end{array}\end{equation}
 Moreover,
if $\grL/\grF$ is
Galois the map $\grM\mapsto q(\grM)$ gives  a one-to-one correspondence between graded subfields~
$\grM$ with
${\grF\subseteq \grM\subseteq \grL}$ and subfields~$M$ with
${q(\grF)\subseteq M\subseteq q(\grL)}$ which preserves degrees and
Galois groups  \cite[Proposition 5.1]{hwang-wadsworth:graded-fld}.
In addition, by passage to quotient fields, we have: if $\grL/\grF$ is normal and $\grM$ is a graded field with
$\grF\subseteq \grM\subseteq \grL$  then:
\begin{equation}\label{12}\grM/\grF\text{ is normal if and only if
$\sigma(\grM) = \grM$ for  every }\sigma\in \Aut(\grM/\grF).
\end{equation}



Consider commuting graded subfields  $\grK,\grL,$ of a graded division
algebra $\grD$ which each contain $\grF:=Z(\grD)$.
The {\it compositum} $\grK\cdot \grL$ is the $\grF$-algebra generated
by $\grK$ and $\grL$ in~$\grD$.
Picking  $\grF_0$-bases $\{k_i\}_{i\in I}$ and
$\{\ell_j\}_{j\in J}$ for $\grK$ and $\grL$, respectively,
which consist of homogenous elements, one has
$$
\grK\cdot\grL\,=\,
\textstyle\sum\limits_{i\in I, j\in J}
\grF_0k_i \ell_j,
$$
and hence $\grK\cdot \grL$ is a graded subfield of $\grD$ (since it is
finite-dimensional over $\grF$).
We say that $\grK$ and $\grL$ are {\it linearly disjoint} over
$\grE:=\grK\cap \grL$ if the natural surjection
$\grK\otimes_{\grE}\grL\ra \grK\cdot \grL$ is an isomorphism. In
particular, $\grK$~and~$\grL$ are linearly disjoint over $\grE$ if and
only if ${[\grK\cdot \grL:\grE]=[\grK:\grE][\grL:\grE]}$.

\begin{lem}\label{lem:linear-disjoint} Let $\grD$ be a graded division
algebra over $\grF$. Let $\grK$ and $\grL$ be commuting graded subfields
of $\grD$ containing $\grF$, with  $\grL$ Galois over~${\grK\cap \grL}$.
Then $\grK$ and $\grL$ are linearly disjoint over $\grK\cap \grL$.
\end{lem}

\begin{proof}

Let $\grE:=\grK\cap \grL$. Then,
$$
q(\grE) \, = \, (\grK\cap \grL) \otimes_\grF q(\grF)
\, = \, \big(\grK \otimes_\grF q(\grF)\big) \cap
\big(\grL \otimes_\grF q(\grF)\big) \, = \,
q(\grK) \cap q(\grL).
$$
Moreover, since $q(\grK) \cdot q(\grL)$ is generated over
$q(\grF)$ by $\grK$ and $\grL$, we have
$q(\grK)\cdot q(\grL)=q(\grK\cdot \grL)$.  As $q(\grL)/q(\grE)$ is
Galois,  $q(\grK)$ and $q(\grL)$ are linearly disjoint over $q(\grE)$;
hence,
$$
[\grK\cdot\grL:\grK]\,=\,[q(\grK)\cdot q(\grL):q(\grK)]
\,=\,[q(\grL):q(\grE)]\,=\,[\grL:\grE].
$$
Therefore, $\grK$ and $\grL$ are linearly disjoint over $\grE$.
\end{proof}







We shall also need the following lemma concerning totally ramified extensions:



\begin{lem}\label{tot-ramified.lem}
Let $\grF\subseteq\grK\subseteq\grM$ be graded fields such that
$\grM/\grK$ is totally ramified. If $\grM/\grF$ is normal
$($resp.~ Galois$)$ then $\grK/\grF$ is normal $($resp.~ Galois$)$.
\end{lem}

\begin{proof}
Since $\grM/\grK$ is totally ramified, by Lemma \ref{lem:tot-ram1}
we have ${\grK \, = \, \bigoplus_{\gamma\in \Gamma_{\grK}} \grM_\gamma}$.
As each $\sigma \in \Gal(\grM/\grF)$ is degree-preserving, this
shows that $\sigma(\grK) = \grK$.
If $\grM$ is normal over $\grF$ then by  \eqref{12}, $\grK$ is also
normal over~$\grF$.
If $\grM/\grF$ is tame then by \eqref{table.equ},  $\grK/\grF$~is tame.
Both statements are proved.
\end{proof}








\label{2}

Finally, we  note that for any finite graded field extension
$\grL/\grF$, if $c \in \grL_0$, then since the roots of its
minimal polynomial $m_{q(\grF),c}$ all have degree $0$, the
polynomial has coefficients in $\grF_0$; so,
$m_{q(\grF),c}=m_{\grF_0,c}$.  Hence,
\begin{equation}\label{0-comp-normal}
\text{ if $\grL/\grF$  is normal (resp.~ Galois) then
$\grL_0/\grF_0$ is normal (resp.~ Galois). }
\end{equation}









\subsection{Canonical subalgebras of graded division algebras}
\label{8}\label{6} \label{5} \label{9}\label{19} The following canonical
algebras and their properties were introduced in
\cite{hwang-wadsworth:graded}. Let $\grD$ be a graded division algebra
over its center~$\grF$.
Then $\grD$ has the following canonical subalgebras:
$$\begin{array}{lll}
\grU & = & \grD_0\otimes_{\grF_0}\grF \, =\, \text{the maximal  subalgebra of }\grD\mbox{ unramified over }\grF, \\
\grZ & = & Z(\grD_0)\otimes_{\grF_0}\grF\, =\, \text{the center of }\grU,  \\
\grC & = & \text{the centralizer }C_\grD(\grU), \\
\grE & = & \text{the centralizer }C_\grD(\grZ)\, =\, \grU\otimes_\grZ \grC.  \\
\end{array}
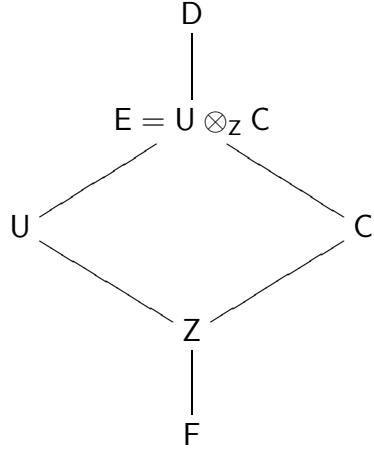
$$

$$\xymatrix{
 & \grD \ar@{-}[d] & \\
 & \grE = \grU\otimes_\grZ \grC \ar@{-}[dr] \ar@{-}[dl] & \\
 \grU \ar@{-}[dr] & & \grC \ar@{-}[dl] \\
 & \grZ \ar@{-}[d] & \\
 & \grF &
}$$

Note that $\grU$ and hence $\grZ$, $\grC$, and $\grE$, were chosen
canonically and hence are invariant under conjugation by nonzero
elements of $\grD_\gamma$, ${\gamma\in \Gamma_\grD}$.
Moreover,  as $\grD$ is totally ramified over $\grU$, one has
$\grU_0=\grE_0=\grD_0$.
In particular, as $[\grD:\grE] = [\grZ: \grF]$ by the graded double
centralizer theorem, we have
\begin{equation}\label{canonical-properties.equ} [\grD:\grE] =
|\Gamma_\grD:\Gamma_\grE| = [\grZ:\grF] = [\grZ_0:\grF_0].
\end{equation}

The center  $Z(\grD_0)= \grZ_0$ clearly contains $\grF_0$ but is not
necessarily equal to it.
In fact, $\grZ_0/\grF_0$ is  Galois with abelian Galois group which is
described as follows:
Let $\int(d_\gamma)$ denote the inner automorphism which sends
$x\in \grD$ to $d_\gamma x d_\gamma^{-1}$.
Let $\theta_\grD\colon\Gamma_\grD\ra \Gal(\grZ_0/\grF_0)$ be the homomorphism
for which $\theta_\grD(\gamma)$ is the restriction of $\int(d_\gamma)$
to $\grZ_0$ for any nonzero $d_\gamma\in \grD_\gamma$. Then
$\theta_\grD$ is well defined, surjective, and its kernel is
$\Gamma_E$, \cite[Proposition 2.3]{hwang-wadsworth:graded}. Hence
$\Gal(\grZ_0/\grF_0)\cong \Gamma_\grD/\Gamma_\grE$.
Note also that as $\grC_0 = C_{\grD_0}(\grD_0) = \grZ_0$, the
graded algebra $\grC$ is totally ramified over $\grZ$.












We shall need the following properties of maximal graded subfields of
$\grC$:

\begin{lem}\label{ramified.lem} Let $\grT$ be a maximal graded subfield
of $\grC$. Then:
\begin{enumerate}
\item  $\grT$  is Galois over $\grF$;
\item   $\Gamma_\grT=\Gamma_{C_\grE(\grT)}$
\end{enumerate}
\end{lem}

\begin{proof}
By the graded double centralizer theorem
\cite[Proposition 1.5]{hwang-wadsworth:graded} and dimension count,
$$
C_\grE(\grT)\,=\,C_\grU(\grZ)\otimes_\grZ C_\grC(\grT)\,=\,
\grU\otimes_\grZ \grT.
$$
As $\Gamma_\grU=\Gamma_\grZ\subseteq \Gamma_\grT$, we get
$\Gamma_{C_\grE(\grT)}=\Gamma_\grT$, 
proving (ii).

To show (i), we first claim that $\grT/\grZ$ is Galois.
Since $\grC$ is totally ramified over~$\grZ$, its graded subfield
$\grT$ is also totally ramified over $\grZ$. 
Hence, by Lemma~\ref{lem:tot-ram1}
\begin{equation}\label{eq:trsubfield}
\grT \, = \, \textstyle \bigoplus\limits_{\gamma\in \Gamma_\grT}
\grC_\gamma.
\end{equation}
\label{4}
By \cite[Proposition 2.3]{hwang-wadsworth:graded}, we have
$\charak\grF_0\ndivides |\Gamma_\grC:\Gamma_\grZ|$, and hence
${\charak\grF_0\ndivides |\Gamma_\grT:~\Gamma_\grZ|}$. This shows that
$\grT$ is tame over $\grZ$. As $\grT$ is also totally ramified over~
$\grZ$, and $\charak\grF_0\ndivides |\Gamma_T:\Gamma_Z|$,
\cite[Proposition 3.3]{hwang-wadsworth:graded-fld}  implies that
$q(\grT)/q(\grZ)$ is a Kummer extension, hence Galois. Thus, $\grT/\grZ$
 is Galois (see \eqref{table.equ}), proving the claim.



As $\grZ$ is Galois and unramified over $\grF$, we have
$\Gal(\grZ/\grF) \cong
\Gal(\grZ_0/\grF_0)$.
Let $\sigma \in \Gal (\grZ/\grF)$.  Since $\theta_\grD$ is
surjective, there is a unit $d\in \grD_\gamma$ for some
$\gamma\in \Gamma$, such that  ${\int(d)|_{\grZ_0}
= \sigma|_{Z_0}}$, and hence  $\int(d)|_\grZ = \sigma$.
Since $d$ lies in one of the components $\grD_\gamma$,  $\int(d)$
preserves $\grC$.
Thus, \eqref{eq:trsubfield} shows that the graded
automorphism $\int(d)$ preserves $\grT$.
Since $\grZ/\grF$ and $\grT/\grZ$ are Galois and since
every automorphism in $\Gal(\grZ/\grF)$ extends to a graded
automorphism of~$\grT$,
it follows that $\grT$ is Galois over $\grF$, as required.
\end{proof}

\subsection{Tame division algebras}
All division algebras considered in this paper are assumed to be
finite-dimensional.
Let $F$ be a Henselian field and $D$ a division algebra with center $F$,
 $\oline D$ the residue division algebra of $D$ with respect to
the unique extension of the valuation on $F$ to $D$, and let
 $\Gamma_D$ be the value group (a totally ordered abelian group).

Recall that $D$ is {\it tame}
(or {\it tamely ramified} over~$F$) if and only if
\begin{equation}\label{defectless}
[D:F]\,=\,[\oline D:\oline F] \,|\Gamma_D:\Gamma_F|,
\end{equation}
$Z(\oline D)$ is separable over $\oline F$, and
${\charak \oline F\ndivides |\ker(\theta_D):\Gamma_F|,}$
cf.~\cite[\S 6]{jacob-wadsworth:div-alg-hensel-field}. Furthermore, $D$~
is said to be  {\it inertial} over $F$ if it is tame and
also unramified, i.e. $\Gamma_D=\Gamma_F$.

The {\it tame Brauer group} $\TBr(F)$ is the subgroup of $\Br(F)$
which consists of classes $[D]$ of  tame division algebras $D$ with center $F$.  Since $D$ is tame
if and only if $D$ is split by the maximal tamely ramified
field extension of $F$, $\TBr(F)$ is a subgroup of $\Br(F)$. Denote the  degree of~$D$ by ${\deg D:=\sqrt{[D:F]}}$,  write $\ind [D]$ for the Schur index $\ind [D]:=\deg D$, and let $[D]^Z=[D\otimes_F Z]\in \Br(Z)$ for any extension $Z/F$.

We shall  need the following lemma from
\cite{jacob-wadsworth:div-alg-hensel-field} which describes properties
that are preserved under tensor products with inertial  algebras.


\begin{lem}\label{JW.lem}
Let $I, D$ be central division algebras over~$F$. Assume $I$ is
inertial and $D$ is tame. Let $D'$ be the division algebra underlying
$[I\otimes_F D]$. Then $Z(\oline{D'})\cong Z(\oline D)$,
$\Gamma_{D'} = \Gamma_{D}$,
$[\oline{D'}]=[\oline{I}\otimes_\oline{F}\oline D]$ in $\Br(Z(\oline D))$,
 and the following ratio is preserved:
\begin{equation}\label{ratio.equ}
\textstyle
\frac{\deg D}{\deg \oline D}  \,=  \,
\frac{\deg D'}{\deg\oline{D'}}.
\end{equation}



\end{lem}

\begin{proof}
All assertions are proved in
\cite[Corollary 6.8]{jacob-wadsworth:div-alg-hensel-field}, except for
the last, which is derived as follows.
Recall that there is a well-defined homomorphism $${\theta_D \colon
\Gamma_D \to\Gal(Z(\oline D)/\oline F)}$$ given as
follows\footnote{This definition slightly differs from the definition
of $\theta_D$ in \cite[p.~133]{jacob-wadsworth:div-alg-hensel-field}, where
$\theta_D$~is defined on $\Gamma_D/\Gamma_F$.}. Let $v$ be the valuation on $D$.
For $\gamma\in \Gamma_D$ take any nonzero  $c\in D$ with
$v(c) = \gamma$. Then, for any $z\in D$ with $v(z) \ge 0$
and $\oline z \in Z(\oline D)$, define
$\theta_D(\gamma)(\oline z) = \oline{c z c^{-1}}$.
By \cite[Corollary 6.8]{jacob-wadsworth:div-alg-hensel-field}, $\theta_{D'}=\theta_{D}$.
Since $D$ is tame, by \eqref{defectless}
$$
\begin{array}{lll}
[D:F]&=&
[\oline D:\oline F]\cdot |\Gamma_D:\Gamma_F| \\
&=& [\oline D:Z(\oline D)]\,[Z(\oline D):\oline F]\,
|\ker\theta_D:\Gamma_F|\,|\image\theta_D|.\end{array}
$$
By \cite[Prop.~1.7]{jacob-wadsworth:div-alg-hensel-field},  $[Z(\oline D):\oline F]=|\image\theta_D|$, hence
by taking square roots we obtain:
\begin{equation}\label{theta.equ}
\deg D\,=\, \deg \oline D \cdot
|\image\theta_D|\cdot \sqrt{|\ker\theta_D:\Gamma_F|}.
\end{equation}
Thus,
$$\textstyle
\frac{\deg D}{\deg \oline D}
\,=\,|\image\theta_D |\cdot \sqrt{|\ker\theta_D:\Gamma_F|}
\,=\,|\image\theta_{D'} |\cdot \sqrt{|\ker\theta_{D'}:\Gamma_F|}
\,=\, \frac{\deg D'}{\deg \oline {D'}}.
$$
\end{proof}

\subsection{The correspondence}

\label{17}


A tame division algebra $D$ with value group $\Gamma$ yields a $\Gamma$-graded division ring
$\gr(D)$ with components $\gr(D)_\gamma=D_{\geq \gamma}/D_{> \gamma},\gamma\in \Gamma,$ where
$$
D_{\geq\gamma} \,=\, \{ x\in D\,|\, v(x)\geq \gamma\}
\text{ and }D_{>\gamma} \,=\, \{ x\in D\,|\, v(x)>\gamma\} .
$$
Furthermore, $\gr(D)$ is a graded division algebra over $\gr(F)$ with
${\gr(D)_0 = \overline D}$, and $\Gamma_{\gr(D)}=\Gamma$. 
Thus, \eqref{defectless.equ} and \eqref{defectless} together show that
\begin{equation}\label{nodefect}
[\gr(D):\gr(F)] \,=\, [D:F].
\end{equation}
Also, the tameness of $D$ implies that ${Z(\gr(D)) = \gr(F)}$ by
\cite[Proposition~4.3]{hwang-wadsworth:graded}.

The map $D\mapsto \gr(D)$ gives a degree-preserving bijection
\cite[Theorem 5.1]{hwang-wadsworth:graded} between tame
division algebras with center $F$ (up to isomorphism) and graded division
algebras with center  $\gr(F)$ (up to isomorphism).
By \cite[Corollary 5.7]{hwang-wadsworth:graded}, this correspondence
is functorial under field extensions $L/F$;  hence, $L$~is a maximal
subfield of $D$ if and only if $\gr(L)$~is a maximal graded subfield
of $\gr(D)$. By \cite[Theorem~1.5]{mounirh-wadsworth:semiramifed}, if
$L/F$ is normal then so is $\gr(L)/\gr(F)$.

On the level of fields,  by
\cite[Theorem 5.2]{hwang-wadsworth:graded-fld}, there is a
correspondence between tame graded field extensions of $\gr(F)$ and tame
field extensions of $F$, which preserves degrees and Galois groups.
In particular, to every tame graded field extension $\grL$ of~$\gr(F)$
there corresponds  a unique tamely ramified
 field extension $L$ of $F$, called
the {\it tame lift} of $\grL$ over~$F$,
such that $\gr(L)\cong \grL$ as graded fields
and ${[L:F]} = {[\grL:\gr(F)]}$.  Moreover,
$L$ is Galois over $F$ if and only if $\grL$ is Galois over over
$\gr(F)$.


\section{Maximal subfields of tame graded division algebras}

Throughout this section we fix a graded division algebra $\grD$ with
center $\grF$, and let $\grZ,\grC,\grU$, and $\grE$ be its canonical
subalgebras (introduced in \S\ref{5}). We first prove the graded
version of Theorem \ref{main.thm}:

\begin{thm}\label{graded-main.thm}
A finite-dimensional graded division algebra $\grD$ has a graded
maximal subfield Galois $($resp.\ normal$)$ over $\grF$ if and only if
$\grD_0$ has a maximal subfield Galois $($resp.\ normal$)$ over $\grF_0$.
\end{thm}


The following Proposition gives the  ``if" implication  of
Theorem \ref{graded-main.thm}.

\begin{prop}\label{prop:gradedmaxsubf}
Let $M$ be a maximal subfield of $\grD_0$, $\grL=M\otimes_{\grF_0}\grF$,
and $\grT$ a maximal graded subfield of $\grC$. Then
$\grM:=\grL\cdot\grT$ is a maximal graded subfield of $\grD$.
Moreover, if $M/\grF_0$ is Galois $($resp.\ normal$)$ then $\grM/\grF$
is Galois $($resp.\ normal$)$.
\end{prop}

\begin{proof}
Since $M$ is maximal it contains $Z(\grD_0)$.
By definition of $\grZ,\grU,\grC$, cf.\ \S\ref{5},
we have $\grZ \subseteq \grL \subseteq \grU$,
and $\grL$ and $\grT$ commute.
Hence, $\grM$ is a graded subfield of $\grD$.
Since $\grL/\grZ$ is inertial we have:
$$
\DIM \grL\grZ \,=\, \DIM {M}{Z(\grD_0)}
\,=\, \deg \grD_0 \,=\, \deg \grU.
$$

As $\grL/\grZ$ is unramified, and $\grT/\grZ$ is totally ramified one
has $\grL\cap\grT=\grZ$. By Lemma~\ref{ramified.lem},  $\grT/\grZ$ is
Galois. Hence, Lemma \ref{lem:linear-disjoint} implies
$$
\DIM\grM\grZ \, = \, \DIM \grL \grZ \cdot \DIM \grT \grZ \, = \,
\deg \grU \cdot \deg \grC  \, = \, \deg \grE.
$$
This shows that $\grM$ is a maximal graded subfield of $\grE$,
cf.\ end of \S\ref{13},
hence also a maximal graded subfield of $\grD$ by
\eqref{canonical-properties.equ}.

Furthermore, if $M=\grL_0$ is Galois (resp.\ normal) over $\grF_0$,
  then $\grL$ is
Galois (resp.\ normal) over $\grF$. As  $\grT$ is Galois over $\grF$
by Lemma \ref{ramified.lem}, we get that  $\grM$ is Galois
(resp.\ normal) over~$\grF$.
\end{proof}


For a maximal graded subfield $\grM$ of $\grD$,
the field $\grM_0$
need not be a maximal subfield of $\grD_0$.
We will therefore modify $\grM$ to enlarge the degree-$0$
part. We start with the following observation:

\begin{lem}\label{maximal-residue.lem}
Let $\grM$ be a maximal graded subfield of $\grD$.
Then, $\grM_0$ is a maximal subfield of $\grD_0$ if and only if
$\grM\supseteq\grZ$ and $|\Gamma_\grM:\Gamma_\grZ|=\deg\grC$.
This holds if $\grM\cap\grC$ is a maximal graded subfield of
$\grC$.
\end{lem}

\begin{proof}
If $\grM_0$ is a maximal subfield of $\grD_0$ then
$\grM_0 \supseteq Z(\grD_0) = \grZ_0$, so
$\grM\supseteq\grZ$
by definition of $\grZ$.
Hence, we assume $\grM\supseteq\grZ$ throughout the proof.
Then $\grM\subseteq C_\grD(\grZ) = \grE$,
so $\grM$ is a maximal graded subfield of $\grE$.
We have
$$
[\grM_0:\grZ_0]\cdot|\Gamma_\grM:\Gamma_\grZ|\,=\,
[\grM:\grZ]\,=\,\deg\grE\,=\,\deg\grU\cdot\deg\grC
\,=\,\deg\grD_0\cdot\deg\grC.
$$
Hence, $\grM_0$ is a maximal subfield of $\grD_0$
(i.e.\ ${[\grM_0:\grZ_0]}=\deg\grD_0$) if and only if
$|\Gamma_\grM:\Gamma_\grZ|=\deg\grC$.

Suppose now that  $\grM\cap\grC$ is a maximal graded subfield of $\grC$.
Then, since $\grM\subseteq C_\grE(\grM\cap\grC)$, by
Lemma~\ref{ramified.lem}
we have
$$
\Gamma_\grM \,\subseteq\, \Gamma_{C_\grE(\grM\cap\grC)}
\,=\,\Gamma_{\grM\cap\grC}\,\subseteq\,\Gamma_\grM,
$$
hence, as $\grM\cap \grC$ is totally ramified over $\grZ$,
 ${|\Gamma_\grM:\Gamma_\grZ|= [\grM\cap \grC:\grZ]=\deg\grC}$.
\end{proof}

The following Proposition gives the  ``only if'' implication  of
Theorem~\ref{graded-main.thm}, and completes its proof.

\begin{prop}\label{max-residue.prop}
Let $\grM$ be a maximal graded subfield of $\grD$
and let $\grT$ be a maximal graded subfield of $\grC$.
Then $\grM':=(\grM\cap C_\grD(\grT))\cdot \grT$ is a maximal graded
subfield of~$\grD$ for which  $\grM'_0$ is a maximal subfield of
$\grD_0$.

Furthermore, if $\grM$ is Galois $($resp.\ normal$)$ over $\grF$ then
$\grM'$ is Galois $($resp.\ normal$)$ over $\grF$  and $\grM'_0$ is
Galois $($resp.\ normal$)$ over $\grF_0$.
\end{prop}

The proof relies on the following lemma:

\begin{lem}\label{lem:AB}
Let $\grA,\grB$ be graded subalgebras of $\grD$ containing $\grF$
such that $\grA$ is a graded field and $\grB \subseteq\grC$.
Define $\grA':=\grA\cap C_\grD(\grB)$.
Then, $\grA/\grA'$ is totally ramified and
$[\grA:\grA']\leq[\grB:\grA\cap\grB]$.
\end{lem}

\begin{proof}

Since $\grB\subseteq\grC$, one has
$U=C_\grD(\grC)\subseteq C_\grD(\grB)$, hence
${C_\grD(\grB)_0=\grU_0=\grD_0}$. Thus,
$\grD$ is totally ramified over $C_\grD(\grB)$.
Hence, (i) $\grA/\grA'$ is totally ramified (for
$\grA'_0=\grA_0\cap C_\grD(\grB)_0=\grA_0\cap\grD_0=\grA_0$),
and (ii) ${\Gamma_{\grA'}=\Gamma_\grA\cap\Gamma_{C_\grD(\grB)}}$ by
Lemma~\ref{lem:tot-ram1}(ii).
Since $\grA$ is a graded field, we have
${\grA\subseteq C_\grD(\grA)\subseteq C_\grD(\grA\cap\grB)}$;
also, $C_\grD(\grB) \subseteq C_\grD(\grA \cap \grB)$.
These together yield,
\begin{multline*}
[\grA:\grA']=|\Gamma_\grA:\Gamma_{\grA'}|=
|\Gamma_\grA:\Gamma_{\grA}\cap\Gamma_{C_\grD(\grB)}|=
|\Gamma_\grA+\Gamma_{C_\grD(\grB)}:\Gamma_{C_\grD(\grB)}|\\
\leq|\Gamma_{C_\grD(\grA\cap\grB)}:\Gamma_{C_\grD(\grB)}|
\leq[C_\grD(\grA\cap\grB):C_\grD(\grB)]=[\grB:\grA\cap\grB],
  \end{multline*}
with the last equality given by the graded double centralizer theorem.
\end{proof}

\begin{proof}[Proof of Proposition \ref{max-residue.prop}]
By Lemma \ref{ramified.lem}, $\grT$ is Galois over $\grF$. Thus, by
Lemma \ref{lem:linear-disjoint}, $\grT$ is linearly disjoint from
$\grM\cap C_\grD(\grT)$
over their intersection $(\grM\cap C_\grD(\grT))\cap \grT = \grM\cap\grT$.
 Hence, by definition of $\grM'$, we have
${[\grM':\grM\cap C_\grD(\grT)]=
[\grT:\grM\cap\grT].}$
Lemma \ref{lem:AB}, applied with $\grA=\grM$ and $\grB=\grT$,
states that this dimension is $\geq[\grM:\grM\cap C_\grD(\grT)]$,
so $[\grM':\grF] \ge [\grM:\grF]$.
Since  $\grM$ is maximal it follows that $\grM'$ is maximal.
Moreover, since $\grM'$ contains $\grT$,
$\grM'_0$ is a maximal subfield of $\grD_0$ by
Lemma~\ref{maximal-residue.lem}.

Assume that $\grM$ is Galois (resp.\ normal) over $\grF$.
The application of Lemma \ref{lem:AB} above also showed that
$\grM$ is totally ramified over $\grM\cap C_\grD(\grT)$.
Hence, by Lemma \ref{tot-ramified.lem},
$\grM\cap C_\grD(\grT)$ is Galois (resp.\ normal) over $\grF$.
Since, by Lemma \ref{ramified.lem}, $\grT$ is Galois over~$\grF$ , we
get that $\grM'$ is Galois (resp.\ normal) over $\grF$.
Hence, $\grM'_0$ is Galois (resp.\ normal) over $\grF_0$, by~
\eqref{0-comp-normal}.
\end{proof}
\begin{rem} \begin{enumerate} \item For a graded subfield $\grM$ of $\grD$ which is not
necessarily maximal, the proof gives $[\grM':\grF]\geq[\grM:\grF]$,
where $\grM':=(\grM\cap C_\grD(\grT))\cdot \grT$.
\item The proof shows that Propositions \ref{prop:gradedmaxsubf} and \ref{max-residue.prop}, and hence also Theorem~\ref{graded-main.thm} hold more generally when  $\grF$ is a proper subfield of $Z(\grD)$ under the assumption that $\grT/\grF$ is Galois.
\end{enumerate}
\end{rem}

Theorem \ref{main.thm} follows from the following corollary:

\begin{cor}\label{cor:tamemaxsubf}
Let $F$ be a Henselian field, and let $D$ be a tame division algebra
with center $F$.
The following are equivalent:
\begin{enumerate}
\item[(a)]
$D$ has a maximal subfield Galois over $F$.
\item[(b)]
$\ovl D$ has a maximal subfield Galois over $\ovl F$.
\item[(c)]
$D$ has a maximal subfield Galois and tamely ramified over $F$.
\end{enumerate}
Moreover, the list can be extended by the three conditions
$(a'),(b'),(c')$
which are obtained from $(a),(b),(c)$ by replacing \lq Galois' with
\lq normal'.
\end{cor}

\begin{proof}[Proof of Corollary \ref{cor:tamemaxsubf}]
Trivially, $(a),(b),(c)$ imply $(a'),(b'),(c')$ respectively, and
$(c')$ implies $(a')$.
Nontrivially, by  \cite[Lemma~3]{saltman:noncr-prod-small-exp},
$(a')$ implies~$(a)$.
Since $D$ is tame, $Z(\ovl D)/\ovl F$ is separable.
Therefore, as was noted in \cite[Prop.~14.2, p.~59]{hanke:thesis},
also $(b')$ implies $(b)$.
We will show $(b)\Rightarrow(c)$ and $(a')\Rightarrow(b')$,
then the proof is completed:
\begin{equation*}
  \xymatrix{ (b) \ar@{<=>}[d] \ar@{=>}[r] & (c) \ar@{=>}[d] & (a) \ar@{<=>}[d] \\
(b')  & (c') \ar@{=>}[r] & (a')  \ar@/^1pc/@{=>}[ll]
}\end{equation*}

$(b)\Rightarrow (c)$:
Suppose $\ovl D = \gr(D)_0$ has a maximal subfield $M$ Galois  over
$\ovl F= \gr(F)_0$.
By Proposition~\ref{prop:gradedmaxsubf},
$\gr(D)$~has a maximal graded subfield $\grM$
that is Galois  over $Z(\gr(D))$.  But $Z(\gr(D)) = \gr(F)$,
as $D$~is tame.
Let $M'$ be the tame lift of $\grM$ over $F$ (cf.\ \S\ref{17}),
i.e.\ the unique tame Galois extension of $F$ with $\gr(M')=\grM$.
By the functoriality mentioned in \S\ref{17}, this $M'$ is a maximal
subfield of $D$, since it splits $D$ and
$$
[M':F] \, = \, [\gr(M'):\gr(F)] \, = \, \deg \,\gr(D) \, = \, \deg D.
$$



$(a')\Rightarrow (b')$:
Let $M$ be a maximal subfield of $D$ that is normal over~$F$.
Then $\gr(M)/\gr(F)$ is normal (cf.\ \S\ref{17}).
As $D$ is defectless over $F$, i.e., equality \eqref{nodefect} holds,
$M$ must also be defectless over $F$.  Thus,
\begin{equation*}\label{eq:degrees}
\DIM{\gr(M)}{\gr(F)}\, = \, \DIM MF \,=\, \deg D \, = \, \deg\,\gr(D),
\end{equation*}
showing that $\gr(M)$ is a maximal graded subfield of $\gr(D)$.
By Proposition \ref{max-residue.prop}, $\gr(D)$ has a maximal graded
subfield $\grM'$ normal over $\grF$ and such that $\grM'_0$ is a maximal
subfield of  ${\ovl D= \gr(D)_0}$. By \eqref{0-comp-normal},
$\grM'_0/F$ is normal.
\end{proof}

\section{Tamely ramified noncrossed products}


\subsection{Simple residue fields}

It is a fundamental question  to determine which division algebras
over a given field $F$ are crossed products.  As a corollary to
Theorem~    \ref{main.thm} we obtain an answer when $F$ is Henselian
and division algebras over the residue field
$K:=\oline F$ are sufficiently well understood.
Let $\cd G_K$ denote the cohomological dimension of the absolute
Galois group $G_K$ of $K$.



\begin{cor}\label{cor:positive} Let $F$ be a Henselian field whose
residue field $K$ is a  local field\,\footnote{We call a field
{\it local} if it is a finite extension of $\mQ_p$
or~$\mathbb{F}_p((t))$ for some prime $p$.}, real closed field,
or satisfies $\cd G_K\leq 1$, then every tame central division algebra
over $F$ is a crossed product.
\end{cor}

\begin{proof} By Theorem \ref{main.thm}, it suffices to show that
$\oline{D}$ has a maximal subfield which is Galois over
$K$. If $\cd G_K\leq 1$, then $\oline D=Z(\oline D)$ is
a field which, since $D$ is tame, is Galois over $K$.  If $K$ is
a real closed field, then either $K$ or $K(\sqrt{-1})$ is a maximal
subfield of $\oline D$ which is Galois over $K$. If $K$ is a local
field then $Z(\oline D)$ has  extensions of arbitrary degree which
are Galois over $K$, simply by composing~$Z(\oline D)$ with
unramified extensions of $K$. This gives the desired result since over
local fields every field of degree $\deg \oline D$ over $Z(\oline D)$
is a maximal subfield of $\oline D$.
\end{proof}

Note that (1) if $K$ is real closed the assertion can be proved
directly without using Corollary \ref{cor:positive}; (2) examples of
fields $K$ for which ${\cd G_K\leq 1}$ include finite fields, and by
Tsen's theorem \cite[\S 19.4]{pierce:ass-alg}, function fields of
curves over algebraically closed fields.


\subsection{Global residue fields}\label{sec:global-residue}

Let $\Gamma$ be the value group of the Henselian
 valuation on $F$. We consider
next  the simplest residue field $K:=\oline F$ for which noncrossed
products  exist over $F$, namely when  $K$ is a global field  \cite{brussel:noncr-prod}.

The tame Brauer group $\TBr(F)$ is described by a generalized Witt
theorem~
\cite[Proposition~3.5]{aljadeff-sonn-wasdworh:projective-schur}\footnote
{The decomposition of $\TBr(F)$ is
described in
\cite{aljadeff-sonn-wasdworh:projective-schur} on the level of
primary components.} as a direct sum:
\begin{equation}\label{equ:gen-witt}
\TBr(F)\,\cong \,\Br(K)\oplus \Hom(G_K, \Delta/\Gamma) \oplus T,
\end{equation}
where $\Delta$ is the divisible
hull of $\Gamma$, and $T$ is a subgroup consisting of classes of some
totally ramified division algebras. Moreover, the subgroup of
$\TBr(F)$ corresponding to ${\Br(K)\oplus \Hom(G_K,\Delta/\Gamma)}$
(resp.~$\Br(K)$) is the subgroup of classes of inertially split
division algebras (resp.~  inertial  division algebras), as described
by a generalization of  Witt's theorem, see  \cite[Satz~2.3]{scharlau:br-henselkoerper}
or \cite[(5.4), Th.~5.6]{jacob-wadsworth:div-alg-hensel-field}.






For  fixed $\chi\in \Hom(G_K,\Delta/\Gamma)$ and $\eta\in T$, we call
the preimage of $\chi+\eta$ under \eqref{equ:gen-witt} the {\it fiber
over}
$\chi+\eta$.
Note that the isomorphism \eqref{equ:gen-witt} is not entirely canonical and a different choice will give us a different fiber.
However, none of our results depends on this choice.
To describe the location of noncrossed products in
$\TBr(F)$, we ask for which  $\chi$ and $\eta$ the fiber over $\chi+\eta$
contains noncrossed products?
This problem was answered in \cite{hanke-sonn:location}
and \cite{hns:existence-bounds} for the inertially split subgroup,
i.e. when $\eta=0$. In the following
we combine Theorem \ref{main.thm} with the methods of
\cite{hanke-sonn:location} and \cite{hns:existence-bounds} to answer
this problem for the entire group $\TBr(F)$.

To this end, we fix $\chi$ and $\eta$ and let $\fC\subseteq\TBr(F)$ be the fiber over
$\chi+\eta$.
For any $c\in\TBr(F)$ we write $\ovl c$ (resp.\ $Z(\ovl c)$) for the class (resp.\ the center) of the residue algebra of the division algebra in $c$.
By Lemma \ref{JW.lem}, $Z:=Z_\fC:=Z(\oline c)$ and
$\ind c/\ind \oline c$ are independent of the choice of $c \in\fC$.

Note that $Z/K$ is abelian, as division algebras in $\fC$ are tame.
If $Z/K$~is cyclic, we say that it is of {\it infinite height}
if for every integer $m$, $Z/K$ embeds into a cyclic extension
$L/K$ with $[L:Z]=m$.
We will prove

\begin{thm}\label{main.cor} Let $Z'/K$ be the maximal subextension of
$Z/K$ of order prime to $\charak K$. Then $\fC$ consists of
crossed products if and only if $Z'/K$ is cyclic of infinite height.
\end{thm}

Moreover, we will show that if $\fC$ contains one noncrossed product
it contains infinitely many of them.

Theorem \ref{main.cor} is already  known if $\fC$ consists of inertially
split division algebras by \cite{hanke-sonn:location, hns:existence-bounds}. Furthermore, for such $\fC$, \cite{hanke-sonn:location, hns:existence-bounds} prove the existence of index bounds which
essentially separate crossed and noncrossed products within the fiber. We do not know
if such bounds exist in fibers which are not inertially split.
Nevertheless, we prove that unless $Z'/K$ is cyclic of infinite height there is a number $m$ (depending only on $Z'$)
such that $\fC$ contains noncrossed products of every index divisible by $m$, see Remark~\ref{bounds.rem}.


\subsection{Residue classes, Galois covers, and their local degrees}

For $m\in\mathbb{N}$,  let $\fC_m$ be the set of $c \in\fC$ with
$m\divides \ind c$,
and $\oline\fC_m$  (resp.~ $\oline\fC$) the set of residue classes of
$\fC_m$ (resp.~ $\fC$).
Note that by Lemma \ref{JW.lem},
$\oline{\alpha+c}=\alpha^Z+\oline c\in \Br(Z)$ for all
$\alpha\in \Br(K),c\in \Br(F)$, where $\alpha+c$ is defined via \eqref{equ:gen-witt}. Hence, for any $\beta\in\oline\fC_m$,
\begin{equation}\label{residue.equ}
\oline\fC_m\,=\,
\{\alpha^Z+\beta : \alpha\in\Br(K), \ind(\alpha^Z+\beta) = m\}.
\end{equation}
By Theorem \ref{main.thm} the following conditions are equivalent:
\begin{enumerate}
\item[($A_m$)] $\fC_m$ consists entirely of crossed products
\item[($\oline{A}_m$)] Every class in $\oline\fC_m$ has a
splitting field $L$, Galois over $K$, with $[L:Z]=~m$.
\end{enumerate}
We call $L\supseteq Z$ an {\it $m$-cover} of $Z/K$ if $L$ is Galois
over $K$ and $[L:Z]=m$. The cover $L$ is {\it cyclic} if $L/K$ is
cyclic. For a prime $\fp$ of $K$, let $K_\fp$ denote the completion
at $\fp$ and let $[L:Z]_\fp:=[L_{\PP'}:Z_{\PP'\cap Z}]$ for any
prime $\PP'$ of $L$ dividing $\fp$.

Let $\beta\in\Br(Z)$. Recall that for a prime $\PP$ of $Z$, the
index $\ind_\PP\beta$ of  $\beta^{Z_\PP}$ equals its exponent
$\exp_\PP\beta$.
By the Albert-Brauer-Hasse-Noether theorem \cite[\S18.4]{pierce:ass-alg},
\begin{equation}\label{ABHN.equ}\beta
\text{ is split by }L\text{ if and only if }
\ind_\PP\beta\,\big|\, [L:Z]_\PP\text{ for every prime }\PP
\text{ of }Z. \end{equation} This allows to translate $(\oline{A}_m)$
into  conditions on the local degrees of covers, as follows. Let
$d_\fp(m):=m$ for every finite prime $\fp$ of $K$, $d_\fp(m)=\gcd(m,2)$
for every real prime $\fp$ of $K$ which is unramified in $Z$, and
$d_\fp(m):=1$ otherwise.

Let $S$ be a finite set of primes of $K$. An $m$-cover $L$ of
$Z/K$ has {\it full local degree in~$S$} if $[L:Z]_\fp=d_\fp(m)$ for
all $\fp\in S$.

\begin{prop}\label{equiv.prop} Assume $\fC_m\neq\varnothing$. There
exists a finite set $T$ of primes of $K$ such that
$(B_m)\Rightarrow (\oline{A}_m)\Rightarrow (B'_m)$, where $(B_m)$
and $(B'_m)$ are the following conditions:
\begin{enumerate}
\item[$(B_m)$] for every $S$, $Z/K$ has an $m$-cover with full
local degree in $S$.
\item[$(B'_m)$] for every  $S$ disjoint from $T$, $Z/K$ has an
$m$-cover  with full local degree in~$S$.
\end{enumerate}
\end{prop}

\begin{proof} $(B_m)\Rightarrow (\oline{A}_m)$:  Let
$\beta\in\oline\fC_m$, and $S$ the set of primes $\fp$ of $K$ such
that the restriction of $\beta$ to $Z_\PP$ is nontrivial for some
prime $\PP\divides \fp$ of $Z$. Applying $(B_m)$ to $S$, we obtain
an $m$-cover $L$ of $Z$ with full local degree in $S$.
By \eqref{ABHN.equ}, $L$ splits $\beta$, as required.

$(\oline{A}_m)\Rightarrow (B'_m)$:
For every $p\,\big|\, [Z:K]$, let $\fp_1^{(p)},\fp_2^{(p)}, \ldots$
be any enumeration of the primes of $K$  so that
$p^n\,\big|\, [Z:K]_{\fp^{(p)}_{i+1}}$ implies
$p^n\,\big|\, [Z:K]_{\fp^{(p)}_i}$ for all $i,n\in\mathbb{N}$.
Let $T$ be the set
$\{\fp_1^{(p)},\fp_2^{(p)}\,|\, p\text{ divides } m\}$.

Let $S$ be  disjoint from $T$ and $\beta\in\oline\fC_m$. Define
$S'$ to be the subset of primes $\fp\in S$ for which
$\ind_\PP\beta<d_\fp(m)$ for all $\PP\divides \fp$. Since $S$ is
disjoint from $T$, we can apply \cite[Lemma 2.5]{hns:existence-bounds}
to obtain a class $\alpha\in \Br(K)$ such that $\ind\alpha^Z=m$,
and $\ind_\PP\alpha^Z=d_\fp(m)$  for all $\fp\in S'$ and all $\PP\divides \fp$.
Furthermore, the proof of \cite[Lemma 2.5]{hns:existence-bounds} gives
$\ind_\PP\alpha^Z=1$ for all $\PP\divides\fp$ with $\fp\in S\setminus S'$.

Let $\gamma:=\alpha^Z+\beta$. Since
$\exp_\PP\beta<\exp_\PP\alpha^Z=d_\fp(m)$ for all
$\PP\divides\fp$, $\fp\in S'$,
we have $\ind_\PP\gamma=\exp_\PP\gamma=d_\fp(m)$ for all
$\PP\divides \fp,\fp\in S'$.  For every $\fp\in S\setminus S'$ there
is a prime  $\PP\divides \fp$, such that  $\ind_\PP\beta=d_\fp(m)$,
and hence, as $\ind_\PP\alpha^Z=1$, one has  $\ind_\PP\gamma=d_\fp(m)$.
By enlarging~$S$, we may assume that $S'$ contains a finite prime $\fp$ and
hence that $\ind\gamma=\ind_\PP\gamma = m$, where $\PP\divides \fp$.

By applying $(A_m)$ to $\gamma$, we obtain an $m$-cover of $Z/K$ which
splits~$\gamma$. Thus, by~ \eqref{ABHN.equ},  $L$ has full local degree
in $S$, as required.
\end{proof}

\begin{rem}\label{negation.rem} Assume $\fC_m\neq\varnothing$. The proof
reveals that  if $(B'_m)$ fails then there are in fact infinitely many
noncrossed products in $\fC_m$.  Indeed if $(B'_m)$ fails for $S$, it
fails for every set $S'$ which contains $S$ and is disjoint from $T$.
Since  $\gamma\in \oline\fC_m$ was constructed with  nontrivial
completions at all primes of  $S'$, there are infinitely many
$\gamma\in\oline\fC_m$ for which $(\oline A_m)$ fails. Thus, there are
infinitely many noncrossed products in $\fC_m$.
\end{rem}

\subsection{Proof of Theorem  \ref{main.cor}}

Set $\ell=\charak(K)$. We use the following lemmas:

\begin{lem}$($\cite[Lemma 2.12]{hns:existence-bounds}$)$\label{previous.lem}
There is a finite set $S_0$ such that every $p^n$-cover $L$ with full
local degree in $S_0$ has abelian kernel $A=\Gal(L/Z)$, for which the
conjugation action of $\Gal(Z/Z')$ on $A$ is trivial.
\end{lem}



\begin{lem}\label{subcover.lem} Let $p\neq \ell$ be a prime. There exists
a finite set $S_0$ disjoint from $T$ such that every  $p^n$-cover $L$
of $Z/K$ with full local degree in $S_0$ contains a $p^n$-cover $L'$
of $Z'/K$.

Furthermore, if $L$ has full local degree in a set $S\supseteq S_0$
then so does~  $L'$.
\end{lem}

\begin{proof}
Let $S_0$ be as in Lemma \ref{previous.lem} and let $B_\ell=\Gal(Z/Z')$.
Let ${A_\ell=\Gal(L/Z')}$
Since $|A|$ and $|B_\ell|$ are relatively prime,
the group extension
$$
1\, \longrightarrow \, A \, \longrightarrow \,
A_\ell\, \longrightarrow \, B_\ell\, \longrightarrow \, 1
$$
is split by the Schur-Zassenhaus theorem.
Since $B_\ell$ acts trivially on~$A$,  ${A_\ell=A\oplus \hat B_\ell}$ with
$\hat B_\ell\cong B_\ell$.
In particular, $A_\ell$ is abelian. Letting ${G=\Gal(L/K)}$ and
$B=\Gal(Z'/K)$
 the group extension ${1\to A_\ell\to G\to  B\to 1}$  induces an action
of $B$ on $A_\ell$.
Being a characteristic subgroup of $A_\ell$, $\hat B_\ell$ is $B$-invariant,
hence normal in $G$.
The fixed field $L'\subseteq L$ of $\hat B_\ell$ is then a cover of
$Z'/K$ with associated group extension
$$
1\, \longrightarrow \, A_\ell/\hat B_\ell\, \longrightarrow \, G/\hat B_\ell
\, \longrightarrow \, B\, \longrightarrow \, 1.
$$
Since $A_\ell/\hat B_l\cong A$ it is a $p^n$-cover.

Full local degree in $S$ is inherited from $L$ since the completions
of $L'$ and of $Z$ are linearly disjoint over a completion of $Z'$.
\end{proof}



\begin{proof}[Proof of Theorem \ref{main.cor}]
Fix $\beta\in\oline\fC$ and let $m=\ind\beta$. By \eqref{residue.equ}
we see that $\fC_{m'}\neq \varnothing$ for every $m\divides m'$.
For a prime $p\neq \ell$, let $p^{s_p}$ (resp.~ $2^{r_2}$ if $p=2$) denote
the number of $p$-power (resp.~ $2$-power) roots of unity in~$Z$
(resp.~ in $Z(\sqrt{-1})$).

Assume $Z'/K$ is cyclic of finite height or noncyclic. We first claim
that there is a prime $p$ for which  $(B_{p^n}')$ fails for all
sufficiently large $n$. Assume first that $Z'/K$ is noncyclic and
let $p\neq \ell$ be a prime for which the $p$-Sylow subgroup of
$\Gal(Z/K)$ is noncyclic. By
\cite[Proposition 3.3]{hns:existence-bounds}, $(B'_{p^n})$ fails for
all $n>2s_p$ if $p$ is odd and for $n>2(r_2+2)$ if $p=2$.

Assume next that $Z'/K$ is cyclic of finite height, and fix
$m\in\mathbb{N}$ such that $Z'/K$ has no cyclic $m$-cover. We can
further assume that  $m$ is a prime power. Indeed, writing
$m=\prod_{p}p^{n_p}$, $p$ prime, if there are cyclic $p^{n_p}$-covers
for all $p\divides m$, their composite gives a cyclic $m$-cover of
$Z'/K$. Let $m=p^{n_p}
$ and let $n>n_p+s_p+1$.
By \cite[Theorem 6.4]{hanke-sonn:location}\footnote{If $Z'/K$ is
non-exceptional we can choose $n:=n_p+s_p+1$, see
\cite{hanke-sonn:location}.}, there is a set $S$ disjoint from $T$,
such that $Z'/K$ has no $p^n$-cover with full local degree in $S$.
By Lemma \ref{subcover.lem}, there is a finite set $S_0$ such that
$Z/K$ has no $p^n$-cover $L$ with full local degree in $S_0\cup S$.
Hence, $(B'_{p^n})$ fails, proving the claim.

Fix an $n$ for which $(B_{p^n}')$ fails and such that $p^n$ is at
least the largest $p$-power dividing $m$. Letting $m'=\lcm(m,p^n)$, by
Lemma \ref{previous.lem} every $m'$-cover with full local degree in a
set $S\supseteq S_0$ has an abelian kernel and hence contains a
$p^n$-cover with full local degree in $S$. Hence, $(B_{m'}')$~fails.
As $\fC_{m'}\neq\varnothing$, Proposition \ref{equiv.prop}  implies that
$(A_{m'})$ fails.

Conversely, assume  that $Z'/K$ is cyclic of infinite height.
By \cite[Theorem 6.3]{hanke-sonn:location}, for every $m$ prime to $\ell$,
$Z'/K$ has an $m$-cover with full local degree. Taking composites with
$Z$, we get that $(B_m)$ holds for every $m$ prime to $\ell$. By
\cite[Lemma 2.10]{hns:existence-bounds},  $(B_{\ell^n})$ holds for all
$n$. By taking composites, we see that $(B_m)$ holds for all
$m\in\mathbb{N}$. Thus, by Proposition~\ref{equiv.prop}, $(A_m)$ holds
for all~$m$.
\end{proof}

\begin{rem}\label{bounds.rem} Assuming $\fC$ does not consist of crossed
products, the proof reveals that there is $p^{n_p}\in \mathbb{N}$, $p$
prime,  such that $(A_m)$ fails for every $m\in \mathbb{N}$ with
$\fC_m\neq \varnothing$ and $p^{n_p}\divides m$.
By Remark \ref{negation.rem}, for such $m$, $\fC_m$~contains infinitely
many noncrossed products.
If the $p$-Sylow subgroup of $\Gal(Z/K)$ is noncyclic and $p\neq \ell$, we
can choose ${n_p=2s_p+1}$ if $p$ is odd and $n_2=2(r_2+2)+1$ if $p=2$.
If $Z'/K$ is cyclic with no cyclic $p^{k_p}$-cover, we can choose
$n_p=k_p+s_p+2$.
\end{rem}

\bibliographystyle{plain}

\end{document}